\documentclass[a4paper,reqno]{amsart}

\usepackage{amsmath}
\usepackage{amssymb}
\usepackage{amsthm}
\usepackage[backend=biber, giveninits=true, maxbibnames=99]{biblatex}
\usepackage{enumitem}
\usepackage{mathtools}
\usepackage[hidelinks, pdfusetitle]{hyperref}
\usepackage[capitalise]{cleveref}

\DeclarePairedDelimiterX{\ha}[2]{\{}{\}}{#1 \,\delimsize|\, #2}
\DeclarePairedDelimiter{\zj}{(}{)}
\DeclarePairedDelimiter{\abs}{|}{|}
\DeclarePairedDelimiter{\haf}{\{}{\}}

\DeclareMathOperator{\diam}{diam}

\newcommand{\N}{\mathbb N}
\newcommand{\Z}{\mathbb Z}
\newcommand{\R}{\mathbb R}
\newcommand{\Q}{\mathbb Q}
\newcommand{\dimH}{\dim_{\mathrm H}}
\newcommand{\dimUB}{\overline{\dim_{\mathrm B}}}
\newcommand{\Hs}{\mathcal H^s}
\newcommand{\uv}{\mathbf u}
\newcommand{\vv}{\mathbf v}
\newcommand{\mv}{\mathbf m}
\newcommand{\qv}{\mathbf q}
\newcommand{\xv}{\mathbf x}

\theoremstyle{plain}
\newtheorem{Thm}{Theorem}
\newtheorem{Prop}[Thm]{Proposition}
\newtheorem{Cor}[Thm]{Corollary}
\newtheorem{Lemma}[Thm]{Lemma}

\theoremstyle{definition}
\newtheorem{Def}[Thm]{Definition}

\newtheorem{Example}[Thm]{Example}

\numberwithin{Thm}{section}

\crefname{Prop}{Proposition}{Propositions}
\crefname{Thm}{Theorem}{Theorems}
\crefname{equation}{}{}

\setlist[enumerate]{label=(\arabic*), ref=(\arabic*)}

\renewbibmacro{in:}{}
\DeclareFieldFormat{pages}{#1}
\DeclareFieldFormat[article]{volume}{\mkbibbold{#1}}
\DeclareFieldFormat[article, inbook]{title}{\mkbibitalic{#1}}
\DeclareFieldFormat{booktitle}{#1}
\DeclareFieldFormat{journaltitle}{#1\isdot}

\title{Lipschitz surjections between self-similar sets}
\author{Attila Gáspár}
\address{Institute of Mathematics, ELTE Eötvös Loránd University, Pázmány Péter s. 1/C, 1117 Budapest, Hungary \newline\indent
HUN-REN Alfréd Rényi Institute of Mathematics, Reáltanoda u. 13-15, 1053 Budapest, Hungary}
\email{gsprati99@gmail.com}

\begin{document}

\begin{abstract}
    We give a characterization of the existence of Lipschitz surjections between dust-like self-similar sets of the same dimension under the assumption that the similarity ratios are commensurable. Our results apply more generally to graph-directed fractals.
\end{abstract}

\maketitle

\section{Introduction}

Recall that a self-similar set is a compact set $X\subset \R^d$ that satisfies the equality $X=\bigcup_{i=1}^n S_i(X)$, where each $S_i : \R^d \to \R^d$ is a contracting similarity. We say that $X$ is \emph{dust-like} if the sets $S_i(X)$ are pairwise disjoint. The set $\Phi={S_1, \dots, S_n}$ is called the \emph{iterated function system} (IFS) generating $X$.

The following result was proved by Balka and Keleti:
\begin{Thm}[{\cite[Corollary 1.7]{BalkaKeleti}}] \label{BK}
    Let $X$ be a dust-like self-similar set, and let $Y$ be a compact metric space such that $\dimH X > \dimUB Y$, where $\dimH$ is the Hausdorff dimension and $\dimUB$ is the upper box dimension. Then there exists a Lipschitz surjection $X \to Y$.
\end{Thm}

It follows from \cref{BK} that if $X$ and $Y$ are both dust-like self-similar sets, then $\dimH X > \dimH Y$ implies that there exists a Lipschitz surjection $X \to Y$. It is also clear from the properties of the Hausdorff dimension that no such map exists when $\dimH X < \dimH Y$.

However, when $\dimH X = \dimH Y$, only partial results are known. Under the additional assumption that $X$ is homogeneous (i.e. the similarities generating $X$ have the same ratio), it was shown in \cite[Theorem 1.17]{BalkaKeleti} that a Lipschitz surjection $X \to Y$ exists if and only if a bi-Lipschitz map $X \to Y$ exists. Recall that a bi-Lipschitz map is a Lipschitz bijection with a Lipschitz inverse.

A similar result was later given by Ruan and Xiao \cite{RuanXiao}: they proved the same equivalence assuming the homogeneity of $Y$ rather than $X$. Consequently, we have the following equivalence:
\begin{Thm}[{\cite[Theorem 1.17]{BalkaKeleti}} and {\cite[Theorem 1.5(1)]{RuanXiao}}] \label{BK_equiv}
    Let $X$ and $Y$ be dust-like self-similar sets, and assume that either $X$ or $Y$ is homogeneous. If $\dimH X = \dimH Y$, then the following are equivalent:
    \begin{enumerate}
        \item There exists a Lipschitz surjection $X \to Y$.
        \item There exists a bi-Lipschitz map $X \to Y$.
    \end{enumerate}
\end{Thm}

Our goal is to give a characterization of the existence of a Lipschitz surjection between dust-like self-similar sets under an assumption weaker than homogeneity.
\begin{Def}
    A set $R\subseteq \R^+$ is \emph{commensurable} if $\log r_1 / \log r_2 \in \Q$ for every $r_1, r_2 \in R$.
\end{Def}

For an IFS $\Phi$ consisting of similarities, denote by $R(\Phi)$ the set of ratios of the similarities in $\Phi$. Our main result is the following:
\begin{Thm} \label{equiv_ifs}
    Let $X$ and $Y$ be dust-like self-similar sets generated by the IFSs $\Phi$ and $\Psi$, and assume that $\dimH X = \dimH Y = s > 0$. Furthermore, assume that either $R(\Phi)$ or $R(\Psi)$ is commensurable. Then the following are equivalent:
    \begin{enumerate}
        \item \label{equiv_ifs.surj} There exists a Lipschitz surjection $X \to Y$.
        \item \label{equiv_ifs.commensurable} The set $R(\Phi) \cup R(\Psi)$ is commensurable.
    \end{enumerate}
\end{Thm}

A remarkable property of this result is that condition \ref{equiv_ifs.commensurable} is symmetric in $X$ and $Y$. Thus, we have the following, somewhat surprising corollary:
\begin{Cor} \label{surj_reverse}
    Let $X$ and $Y$ be dust-like self-similar sets, and assume that either $X$ or $Y$ has commensurable ratios. If $\dimH X = \dimH Y$ and there exists a Lipschitz surjection $X \to Y$, then there is also a Lipschitz surjection $Y \to X$.
\end{Cor}
We do not know whether \cref{surj_reverse} holds if we do not assume the commensurability of either $X$ or $Y$.

The commensurability assumption was considered previously in the literature concerning the bi-Lipschitz equivalence of self-similar sets. The following result of Xi and Xiong gives a complete characterization when the similarity ratios are commensurable. For an IFS $\Phi$ with commensurable ratios, we can write the ratios of the similarities in the form $r_\Phi^{k_1}, \dots, r_\Phi^{k_n}$, where $r_\Phi \in (0, 1)$ and $\operatorname{gcd}(k_1, \dots, k_n)=1$. We define the \emph{measure root} $p_\Phi=r_\Phi^s$, where $s$ is the similarity dimension of $\Phi$.
\begin{Thm}[{\cite[Theorem 1.1]{XiXiong}}] \label{XiXiong}
    Let $X$ and $Y$ be dust-like self-similar sets generated by the IFSs $\Phi$ and $\Psi$, and assume that $\Phi$ and $\Psi$ both have commensurable ratios. Then the following are equivalent:
    \begin{enumerate}
        \item There exists a bi-Lipschitz map $X \to Y$.
        \item $\dimH X = \dimH Y$, $\log r_{\Phi}/\log r_{\Psi} \in \Q$ and $\Z[p_{\Phi}]=\Z[p_{\Psi}]$.
    \end{enumerate}
\end{Thm}
We remark that in \cite{XiXiongOSC}, Xi and Xiong obtained a generalization of this result for totally disconnected self-similar sets with commensurable ratios that satisfy the open set condition, however, the condition for the equivalence is more complicated.

In \cite{BalkaKeleti}, Balka and Keleti asked whether \cref{BK_equiv} holds without the homogeneity assumption. A counterexample was given by Ruan and Xiao in \cite[Example 6.2]{RuanXiao}. We will now give a simpler counterexample using our characterization in \cref{equiv_ifs}.
\begin{Example}
    Consider the IFSs
    \begin{align*}
        \Phi &= \haf[\Big]{x \mapsto \frac{x}{4}, x \mapsto \frac{1}{2} + \frac{x}{2}}, \\
        \Psi &= \haf[\Big]{x \mapsto \frac{x}{8}, x \mapsto \frac{1}{5} + \frac{x}{8}, x \mapsto \frac{2}{5} + \frac{x}{8}, x \mapsto \frac{3}{5} + \frac{x}{8}, x \mapsto \frac{4}{5} + \frac{x}{64}}
    \end{align*}
    defined on $\R$, and let $X$ and $Y$ be their respective attractors. We can check that $X, Y \subseteq [0, 1]$ and both $\Phi$ and $\Psi$ are dust-like, furthermore, $\dimH X = \dimH Y = \log_2 \varphi$, where $\varphi=\zj[\big]{\sqrt5 + 1}/2$ is the golden ratio. It is clear that the measure roots are $p_\Phi=\varphi^{-1}$ and $p_{\Psi}=\varphi^{-3}$. Notice that
    \[\varphi^{-3}=\sqrt5 - 2 \in \Z[\sqrt5],\]
    but $\Z[\sqrt5]$ does not contain $\varphi^{-1}=\zj[\big]{\sqrt5 - 1}/2$. Therefore, $p_\Phi$ and $p_\Psi$ generate different rings, so $X$ and $Y$ are not bi-Lipschitz equivalent by \cref{XiXiong}. However, the ratios are all integer powers of 2, so there exist Lipschitz surjections in both directions by \cref{equiv_ifs}.
\end{Example}

Rather than considering self-similar sets, we will work with the more general notion of \emph{graph-directed} fractals introduced by Mauldin and Williams \cite{MauldinWilliams}.

\begin{Def}
    A \emph{graph-directed iterated function system} (GD-IFS) is a triple $\Phi=(V, E, S)$ consisting of a finite directed graph $(V, E)$ together with a contracting similarity $S_e : \R^d \to \R^d$ for each edge $e\in E$. The compact sets $\{X_i\}_{i\in V}$ are called the \emph{attractors} if for every $i \in V$,
    \begin{equation}
        X_i = \bigcup_{j\in V} \bigcup_{e \in E_{i, j}} S_e(X_j) \label{eqn:gd_decomp},
    \end{equation}
    where $E_{i,j}$ is the set of edges from $i$ to $j$.
\end{Def}

Analogously with self-similar sets, there exists a unique attractor vector $(X_i)_{i\in V}$ for every GD-IFS. An important special case is when the graph consists of a single vertex and several loop edges. It is easy to check that in this case, the attractor corresponding to the single vertex is the self-similar set generated by the similarities on the loop edges.

Throughout the paper, we will make some assumptions about GD-IFSs. First, we will assume that $\Phi$ is \emph{dust-like}, that is, for every $i \in V$, the sets $S_e(X_j)$ in the attractor property \cref{eqn:gd_decomp} are pairwise disjoint. Second, we will require that the graph $(V, E)$ is strongly connected, that is, there exists a directed path between any two vertices. Under these assumptions, the attractors of $\Phi$ all have the same dimension $s$, furthermore, their $s$-dimensional Hausdorff measure is positive and finite \cite[Theorem 3]{MauldinWilliams}. In particular, this implies that every (relatively) open subset of an attractor has positive Hausdorff measure. Without the strong connectivity assumption, the situation is more complicated, see \cite[Theorem 4]{MauldinWilliams}.

Let $\Phi=(V, E, S)$ be a GD-IFS. For an edge $e \in E$, we will denote by $r_e$ the similarity ratio of $S_e$. For a cycle $C$ in the graph $(V, E)$ (i.e. a cyclic walk without repeated vertices), we can define the \emph{cycle ratio} $r_C = \prod_{e\in C} r_e$. We will denote by $R(\Phi)$ the set of cycle ratios of $\Phi$. In the special case of an IFS, it is clear that $R(\Phi)$ is the set of ratios, matching our earlier definition.

The graph-directed generalization of \cref{equiv_ifs} is the following:
\begin{Thm} \label{equiv}
    Let $X$ and $Y$ be attractors of the dust-like strongly connected GD-IFSs $\Phi$ and $\Psi$, and assume that $\dimH X = \dimH Y = s > 0$. Furthermore, assume that either $R(\Phi)$ or $R(\Psi)$ is commensurable. Then the following are equivalent:
    \begin{enumerate}
        \item \label{equiv.surj} There exists a Lipschitz surjection $X \to Y$.
        \item \label{equiv.commensurable} The set $R(\Phi) \cup R(\Psi)$ is commensurable.
    \end{enumerate}
\end{Thm}
It is important to note that \cref{equiv} is not just for the sake of generalization. The graph-directed case is essential for some parts of the proof, for instance, \cref{homogeneous_power} has no analogue in the usual self-similar setting.

A related result was obtained for graph-directed fractals by Xi and Xiong:
\begin{Thm}[{\cite[Theorem 2]{XiXiongGD}}]
    Let $X$ and $Y$ be attractors of the dust-like strongly connected GD-IFSs $\Phi$ and $\Psi$, and assume that $\dimH X = \dimH Y = s > 0$. Furthermore, assume that every similarity ratio of $\Phi$ is some $r_1 \in (0, 1)$, and every similarity ratio of $\Psi$ is some $r_2 \in (0, 1)$. If $r_1^{-s}$ and $r_2^{-s}$ are both integers, then the following are equivalent:
    \begin{enumerate}
        \item There exists a bi-Lipschitz map $X \to Y$.
        \item $\log r_1 / \log r_2 \in \Q$.
    \end{enumerate}
\end{Thm}
Note that this theorem assumes that every similarity ratio is the same, while in \cref{equiv}, we only assume the weaker property that the ratios of cycles are commensurable. However, as we will see in \cref{commensurable_to_homogeneous}, these are essentially equivalent.

The paper is structured as follows. In \cref{sec:homogeneous}, we will show that the GD-IFSs can be replaced by ones with the same ratio for every similarity, plus some additional technical assumptions. In \cref{sec:measure_linear}, we will generalize some results of Ruan and Xiao \cite{RuanXiao} on measure-linear Lipschitz surjections to the graph-directed setting. Finally, in \cref{sec:surj_to_comm}, we will show the implication $\ref{equiv.surj} \Rightarrow \ref{equiv.commensurable}$, and the implication $\ref{equiv.commensurable} \Rightarrow \ref{equiv.surj}$ will be proved in \cref{sec:comm_to_surj}.

\section{Homogeneous GD-IFSs} \label{sec:homogeneous}

\begin{Def}
    A directed graph is \emph{primitive} if it is strongly connected and the GCD of the lengths of the cycles is 1.
\end{Def}

\begin{Def}
    We call a GD-IFS $\Phi=(V, E, S)$ \emph{homogeneous} if the similarity ratio of every $S_e$ ($e \in E$) is the same. We say that $\Phi$ is \emph{primitive} if it is homogeneous and the graph $(V, E)$ is primitive.
\end{Def}

The main goal of this section is to show that an attractor of a GD-IFS with commensurable cycle ratios can be obtained as the attractor of a primitive GD-IFS.

\begin{Lemma} \label{integer_weight}
    Let $G = (V, E)$ be a strongly connected finite directed graph. Assume that a weight $w : E \to \R$ has the property that every (directed) cycle has positive integer weight. Then for every $v_0\in V$, there exists a potential $p : V \to \R$ such that for every edge $e \in E_{u,v}$, the modified weight $w'(e) = w(e) + p(u) - p(v)$ is a nonnegative integer, furthermore, $w'(e)$ is positive for any edge $e \in E_{u,v_0}$.
\end{Lemma}
\begin{proof}
    First, we claim that the weight of a cyclic walk is a nonnegative integer. To see this, notice that a nonempty cyclic walk contains a cycle. Removing this cycle from the walk, we can proceed by induction on the number of edges.

    Let $p(v)$ be the minimum weight of a walk from $v_0$ to $v$. Clearly $p(v)$ is finite by the nonnegativity of weights of cyclic walks. It is clear from the minimality that for an edge $e \in E_{u,v}$, $p(v) \le p(u) + w(e)$, or equivalently, $w'(e) \ge 0$.

    Now let $e\in E_{u,v}$ be an arbitrary edge. There exist minimum weight walks $P_u$ and $P_v$ from $v_0$ to $u$ and $v$, respectively. It is clear that $w'$ is zero for all the edges in $P_u$ and $P_v$. Choose a walk $Q$ from $v$ to $v_0$. Since modifying the edge weights by a potential does not change the weights of cyclic walks, the cyclic walks $P_u e Q$ and $P_v Q$ both have integer weight with respect to $w'$. Taking the difference gives $w'(e)$, since the weight of $Q$ cancels out, furthermore, $P_u$ and $P_v$ both have zero weight. Thus, we have shown that $w'(e)$ is an integer.

    Finally, assume that $w'(e) = 0$ for some edge $e\in E_{u,v_0}$. Let $P_u$ be a minimum $w$-weight walk from $v_0$ to $u$. Since each edge in $P_u$ has zero $w'$-weight, the cycle $P_u e$ has zero $w'$-weight. Consequently, it has zero $w$-weight, contradicting the positivity assumption.
\end{proof}

\begin{Prop} \label{commensurable_to_homogeneous}
    Let $X$ be an attractor of a dust-like strongly connected GD-IFS $\Phi$ with commensurable cycle ratios. Then there exists a primitive dust-like strongly connected GD-IFS $\Phi'$ that also has $X$ as an attractor. Furthermore, the common similarity ratio $r$ of $\Phi'$ can be chosen such that every element of $R(\Phi)$ is of the form $r^k$ for some positive integer $k$.
\end{Prop}
\begin{proof}
    Since there are finitely many cycles, there exists an $r\in (0, 1)$ such that the elements of $R(\Phi)$ are all of the form $r^k$ for some positive integer $k$. We may assume that the GCD of these exponents is 1, otherwise, we can replace $r$ by some power.

    Define the weight $w(e) = \log_r r_e$. It is clear from the choice of $r$ that every cycle has integer weight. Let $v_0\in V$ be the vertex for which $X=X_{v_0}$. By \cref{integer_weight}, we can choose a potential $p$ such that for every $e \in E_{u,v}$, the weight $w'(e) = w(e) + p(u) - p(v)$ is a nonnegative integer and $w'(e) > 0$ for every $e\in E_{u,v_0}$. Without loss of generality, we may assume that $p(v_0) = 0$. For every $e \in E_{u,v}$, define the similarity
    \[S'_e(x) = r^{p(u)} S_e\zj[\big]{r^{-p(v)} x}.\]
    We can see that the similarity ratio of $S'_e$ is $r^{w'(e)} \le 1$. Let $X'_u = r^{p(u)} X_u$. Clearly $X'_{v_0} = X_{v_0} = X$. We can check that
    \begin{align*}
        X'_u &= r^{p(u)} \bigcup_{v\in V} \bigcup_{e \in E_{u,v}} S_e(X_v) =
        \bigcup_{v\in V} \bigcup_{e \in E_{u,v}} r^{p(u)} S_e(X_v) =\\&=
        \bigcup_{v\in V} \bigcup_{e \in E_{u,v}} S'_e\zj[\big]{r^{p(v)} X_v} =
        \bigcup_{v\in V} \bigcup_{e \in E_{u,v}} S'_e(X'_v).
    \end{align*}
    This shows that the sets $X'_u$ are the attractors of the GD-IFS $\Phi_1 = (V, E, S')$. Note that the ratios along cycles are the same for $S$ and $S'$. Compared to the earlier definition, we must allow $S'$ to have isometries, since the similarity ratio can be 1. Our next goal is to eliminate these, which we will do in a sequence of steps.

    Let $e\in E_{u,v}$ be an edge. We can eliminate $e$ from the graph as follows: for each edge $f\in E_{v,w}$, insert a new edge $f^*$ into $E_{u,w}$, and set $S'_{f^*} := S'_e \circ S'_f$. Next, delete the edge $e$. After this, it is possible that $v$ has no incoming edges, in this case, we will also delete $v$ and all outgoing edges from the graph. We will denote the modified set of edges by $E^*$.

    We will now check that the graph remains strongly connected. For any walk $P$ that does not end in $v$, we can replace each occurrence of $e f$ with the corresponding $f^*$ added when eliminating $e$. This shows that any vertex that is not $v$ can be reached from any vertex. If we did not delete $v$, then it can be reached from some other vertex, hence the graph remains strongly connected.

    For a cycle in the new graph, we can replace each new edge $f^*$ by $e f$. This way, we get a cyclic walk in the original graph that has the same similarity ratio. This shows that the ratios of cycles are still less than 1. It is also clear that the ratio of each edge is still of the form $r^k$ with $k \ge 0$. Furthermore, the GCD of the exponents of $r$ in the cycle ratios will still be 1.

    Next, we will see that the attractors do not change after an elimination step. Since the outgoing edges change only at $u$, we only have to check the attractor property for $X'_u$. Notice that
    \begin{align*}
        S'_e(X'_v) &= S'_e\zj[\Bigg]{\bigcup_{w\in V} \bigcup_{f\in E_{v,w}} S'_f(X'_w)} =
        \bigcup_{w\in V} \bigcup_{f\in E_{v,w}} S'_e\zj[\big]{S'_f(X'_w)} =\\&=
        \bigcup_{w\in V} \bigcup_{f\in E_{v,w}} S'_{f^*}(X'_w).
    \end{align*}
    This implies that the attractor property for $X'_u$ remains true when moving from $E$ to $E^*$. Furthermore, since the sets in the union above are disjoint by the dust-like property at $X'_v$, the dust-like property will still hold for $X'_u$.

    Let us call an edge \emph{isometric} if $S'_e$ is an isometry. We will now perform a sequence of elimination steps to remove the isometric edges. We take a vertex $v$ that has incoming isometric edges, and eliminate these edges in any order. Now suppose that we added an isometric edge $f^* \in E_{u,w}$ in this process. This is possible only if the original edge $f$ was also isometric. It follows that $w$ already had an incoming isometric edge. Therefore, the set of vertices with incoming isometric edges does not increase. After these steps, we have removed all the incoming isometric edges at $v$. We still have to check that we did not add any new isometric edge $f^* \in E^*_{u,v}$. This is only possible if $f \in E_{v,v}$ and $f$ is also isometric. However, this implies that the cycle consisting of just $f$ has similarity ratio $1$, but this is a contradiction, since the invariant that every cycle has ratio less than 1 holds during the whole process.

    We have shown that after performing these elimination steps at a vertex $v$, the set of vertices with incoming isometric edges does not increase, furthermore, $v$ will be removed from this set. It follows that after doing these elimination steps for all vertices, there will be no isometric edges remaining. However, we have to be careful, since if we delete $v_0$ during the process, then $X=X'_{v_0}$ may no longer be an attractor. But this is impossible: by the choice of $p$, $v_0$ initially had no incoming isometric edges, so there was no step where it could have been possibly deleted.

    Let $\Phi_2 = (\tilde V, \tilde E, \tilde S)$ be the GD-IFS obtained after all the elimination steps. The similarities in $\tilde S$ all have ratio $r^k$ for some $k > 0$. The last step of the proof is to replace each edge with ratio $r^k$ by a path consisting of $k$ edges. Specifically, let $e\in \tilde E_{u,v}$ be an edge with ratio $r^k$. We will replace $e$ by the path
    \[u \stackrel{f_1}\to w_1 \stackrel{f_2}\to w_2 \dots w_{k - 1}\stackrel{f_k}\to v.\]
    Set $\tilde X_{w_i} := r^{k - i} X'_v$. For $2 \le i \le k$, let $\tilde S_{f_i}(x) := r x$. Finally, let $S_{f_1}(x) = \tilde S_e(r^{-(k - 1)} x)$. 
    
    Let $\Phi'$ be the GD-IFS obtained by doing this transformation for each edge of $\tilde E$. For the existing vertices $w \in \tilde V$, we can set $\tilde X_w = X'_w$. It is easy to check that the sets $\tilde X_w$ are the attractors of $\Phi'$. It is also clear that each similarity of $\Phi'$ has ratio $r$. Note that there is a bijection between the cycles of $\Phi_2$ and $\Phi'$, and the corresponding cycles have the same ratio. It follows that the GCD of the exponents of $r$ in the cycle ratios of $\Phi'$ is 1. By the homogeneity, this implies that the cycle lengths have GCD 1, thus $\Phi'$ is primitive. Finally, notice that $\tilde X_{v_0} = X'_{v_0} = X$, which completes the proof.
\end{proof}

We will now prove some properties of homogeneous GD-IFSs that will be needed later.

For a homogeneous GD-IFS $\Phi=(V, E, S)$ with attractors $X_1, \dots, X_n$, we can define the $n \times n$ matrix $A$ by $A_{i,j} = \#E_{i,j}$. We will call $A$ the \emph{adjacency matrix} of $\Phi$. We can also define the column vector $\mv \in \R^V$ by $\mv = (\Hs(X_1), \dots, \Hs(X_n))^\top$. Assuming that $\Phi$ is dust-like, it follows from the attractor property \cref{eqn:gd_decomp} that
\begin{equation*}
    \Hs(X_i) = \sum_{j\in V} \sum_{e \in E_{i, j}} \Hs\zj[\big]{S_e(X_j)} = \sum_{j\in V} A_{i,j} r^s \Hs(X_j),
\end{equation*}
or equivalently, $\mv = r^s A \mv$.

\begin{Def}
    A nonnegative square matrix $A$ is \emph{primitive} if there exists a positive integer $k$ such that $A^k$ has positive entries.
\end{Def}

It is well-known (see e.g. \cite[Chapter IV, Theorem 3.3]{PF2}) that a directed graph is primitive if and only if its adjacency matrix is primitive. This has the following corollary:

\begin{Prop} \label{matrix_primitive}
    A homogeneous GD-IFS $\Phi$ is primitive if and only if its adjacency matrix $A$ is primitive.
\end{Prop}

An important property of primitive matrices is the Perron--Frobenius theorem, which states that a primitive matrix has a positive eigenvalue (called the Perron root) for which the corresponding left and right eigenvectors are positive, that is, all their components are positive. Furthermore, if an eigenvector is positive, then the corresponding eigenvalue must be the Perron root. See \cite{PF1} for a reference.

In our case, assuming that $\Phi$ is a dust-like primitive GD-IFS with similarity ratio $r$ and dimension $s$, its adjacency matrix $A$ is primitive and $A \mv = r^{-s} \mv$. Since $\mv$ is positive, it follows that $r^{-s}$ is the Perron root of $A$.

For a walk $P=e_1 e_2 \dots e_n$ ending in $j\in V$, we can define the similarity $S_P=S_{e_1} \circ S_{e_2} \circ \dots \circ S_{e_n}$ and the set $X_P = S_P(X_j)$. It is clear that $X_P$ is a copy of $X_j$ scaled by $r_P = r_{e_1} \cdots r_{e_n}$. Let $E_{i, j, l}$ denote the set of walks of length $l$ from $i$ to $j$. We will call the sets of the form $X_P$ for $P\in E_{i, j, l}$ the level $l$ cells in $X_i$. It is clear that for any $i\in V$ and $l \ge 0$, we have
\[X_i = \bigcup_{j\in V} \bigcup_{P \in E_{i, j, l}} X_P,\]
furthermore, the sets $X_P$ in the union above are pairwise disjoint.

\begin{Prop} \label{homogeneous_power}
    Let $X$ be the attractor of a dust-like primitive GD-IFS with similarity ratio $r$. Then for every positive integer $k$, $X$ can be obtained as an attractor of a dust-like primitive GD-IFS with similarity ratio $r^k$.
\end{Prop}
\begin{proof}
    Let $X_1, \dots, X_n$ be the attractors of $\Phi=(V, E, S)$. We can take the decomposition
    \[X_i=\bigcup_{j\in V}\bigcup_{P \in E_{i,j,k}} S_P(X_j).\]
    This shows that $X_1, \dots, X_n$ are the attractors of the GD-IFS $\Phi'$ where the edges correspond to the walks of length $k$ in $(V, E)$ and the similarities are of the form $S_P$. It is clear that $S_P$ has similarity ratio $r^k$, which implies that $\Phi'$ is homogeneous. Furthermore, we can see that the sets $S_P(X_j)$ in the decomposition above are pairwise disjoint, thus $\Phi'$ is dust-like. Finally, it is clear that if $\Phi$ has adjacency matrix $A$, then the adjacency matrix of $\Phi'$ is $A^k$, which is primitive if $A$ is primitive.
\end{proof}

\begin{Prop} \label{vector_in_ring}
    Let $\Phi$ be a dust-like primitive GD-IFS with similarity ratio $r$ and Hausdorff dimension $s$, and let $\mv$ be its measure vector. Let $p=r^s$. Assume that some component of $\mv$ is in $\Q[p]$. Then every component of $\mv$ is in $\Q[p]$.
\end{Prop}
\begin{proof}
    Without loss of generality, we may assume that $\mv_1 \in \Q[p]$. Consider the system of linear equations
    \[p A \xv = \xv; \quad \xv_1 = \mv_1.\]
    Note that $\xv=\mv$ is a solution of this system of equations. Since all the coefficients are in the field $\Q(p)$ and there is a solution over $\R$, it follows that there is a solution $\xv$ with components in $\Q(p)$. Note that $p^{-1}$ is the Perron root of $A$, so the associated eigenspace is one-dimensional. Combined with the equation $\xv_1=\mv_1 > 0$, this implies that $\xv=\mv$. This shows that every component of $\mv$ is in $\Q(p)$. Finally, since $A$ has integer entries, it is clear that its eigenvalue $p^{-1}$ must be algebraic. Consequently, $p$ is also algebraic, thus the ring $\Q[p]$ and the field extension $\Q(p)$ coincide.
\end{proof}

\section{Measure-linear maps} \label{sec:measure_linear}

Let $X$ and $Y$ be attractors of the dust-like strongly connected GD-IFSs $\Phi = (V^X, E^X, S^X)$ and $\Psi = (V^Y, E^Y, S^Y)$ with the same Hausdorff dimension $s > 0$. In this section, we will show that if a Lipschitz map $f : X \to Y$ with $\Hs(f(X)) > 0$ exists, then there is also some Lipschitz map that satisfies a measure-linearity property. Note that in this section, we will not assume any commensurability conditions on $\Phi$ and $\Psi$.

The proof in this section is based on \cite{RuanXiao} with some changes for the graph-directed setting. Note that we will not use \cite[Proposition 3.4]{RuanXiao}, instead, we will give a simpler proof using the martingale-based approach of Wang \cite{Wang}.

First, we define some constants. Denote by $r_{\min}^X$ and $r_{\min}^Y$ the smallest contraction ratio in $\Phi$ and $\Psi$, respectively. Set
\[\delta = \min_{\substack{u, v \in V^Y \\ e \in E^Y_{u,v}}} \frac{d(Y_e, Y_u \setminus Y_e)}{r_e}\]
and
\[\tilde L = \min_{u \in V^X} \frac{\delta}{\diam X_u}.\]
The choice of these constants is motivated by the following lemma:
\begin{Lemma} \label{preimage_decomp}
    Let $u \in V^X$, $v \in V^Y$, $r_0 \in [r_{\min}^X, 1]$, and assume that $f : r_0 X_u \to Y_v$ is an $\tilde L$-Lipschitz map. For an arbitrary walk $P$ starting from $v$, $f^{-1}(Y_P)$ can be written as a finite disjoint union of cells, each of them congruent to some $r_Q X_w$, where $r_{\min}^X r_P \le r_0 r_Q < r_P$.
\end{Lemma}
\begin{proof}
    Consider a walk $Q$ from $u$ to some $w$ such that $r_Q < r_P$. By the choice of $\tilde L$, we have
    \[\diam f(X_Q) \le \tilde L \diam X_Q = \tilde L \diam X_w r_Q < \delta r_P.\]
    It is easy to see from the choice of $\delta$ that $d(Y_P, Y_v \setminus Y_P) \ge \delta r_P$. It follows that either $f(X_Q) \subseteq Y_P$ or $f(X_Q) \cap Y_P = \emptyset$, or equivalently, either $X_Q \subseteq f^{-1}(Y_P)$ or $X_Q \cap f^{-1}(Y_P) = \emptyset$.
    
    If we take $Q$ to be minimal with the property $r_Q < r_P$, then the inequality $r_Q \ge r_P r_{\min}^X$ also holds, otherwise, we could remove the last edge of $Q$, contradicting the minimality. These minimal walks partition $X_u$ into finitely many cells, so $f^{-1}(Y_P)$ can be written as the disjoint union of some of them.
\end{proof}

Let $\mathcal X$ be the family of extended metric spaces $\tilde X$ which can be written as the disjoint union
\[\tilde X = \bigcup^*_{1 \le i \le p} c_i X_{v_i},\]
where $c_i \in [r_{\min}^X, 1]$ and $v_i \in V^X$. On each $c_i X_{v_i}$, we take the metric on $X_{v_i}$ scaled by $c_i$. For $i \ne j$, we take the distance between points of $c_i X_{v_i}$ and $c_j X_{v_j}$ to be infinite.

We give the following generalization of \cite[Lemma 4.1]{RuanXiao}. Note that in this proposition, we do not assume any commensurability condition.
\begin{Prop} \label{multipart_surj}
    Assume that there exists a Lipschitz map $f : X \to Y$ such that $\Hs(f(X)) > 0$. Then for each $u \in V^Y$, there exists an $\tilde X_u \in \mathcal X$ and an $\tilde L$-Lipschitz surjection $\tilde f_u : \tilde X_u \to Y_u$. Furthermore, $\tilde X_u$ can be chosen such that $\Hs(\tilde X_u)$ is minimal among all such $(\tilde X_u, \tilde f_u)$ pairs.
\end{Prop}

We will use the following generalization of \cite[Lemma 3.4]{RuanXiao}. We omit the proof, since the generalization to the graph-directed setting is straightforward. Denote by $\mathcal B_{l}$ the (finite) $\sigma$-algebra consisting of the unions of some level $l$ cells in $Y$.
\begin{Lemma}
    Let $h_l : Y \to \R$ be the $\mathcal B_l$ measurable function such that 
    \[h_l|_{Y_P}\equiv \frac{\Hs\zj[\big]{Y_P\cap f(X)}}{\Hs(Y_P)}\]
    for every $i\in V^Y$ and $P \in E^Y_{1,i,l}$. Then $h_l(y) \to 1$ for $\Hs$-a.e. $y \in f(X)$.
\end{Lemma}

We can also define the $\mathcal B_l$ measurable function $g_l : Y \to \R$ such that
\[g_l|_{Y_P}\equiv \frac{f_*\Hs(Y_P)}{\Hs(Y_P)}\]
for every $i\in V^Y$ and $P \in E^Y_{1,i,l}$. The following proposition is an easy consequence of the dust-like property of $\Psi$:
\begin{Prop}
    The functions $g_l : Y \to \R$ are a martingale with respect to the filtration $\mathcal B_l$ and the probability measure $\Hs|_Y/\Hs(Y)$.
\end{Prop}

Fix some $u \in V^Y$. For $t \in [0, 1]$, let
\[
    m_u(t) = \inf \ha[\big]{\Hs(\tilde X)}{\tilde X \in \mathcal X, \text{$\tilde f : \tilde X \to Y_u$ is $\tilde L$-Lipschitz}, \Hs\zj[\big]{\tilde f(\tilde X)} \ge t \Hs(Y_u)}.
\]
It is clear that $m_u$ is increasing.

\begin{Prop} \label{ae_cycle}
    For every cyclic walk $C$ in $(V^Y, E^Y)$ and $\Hs$-a.e. point $y \in Y$, the infinite walk $P$ encoding $y$ contains $C$.
\end{Prop}
\begin{proof}
    Since the graph is strongly connected, there exists a $k$ such that for every $i \in V^Y$, there is a walk of length $k$ starting from $i$ that contains $C$. It follows that there is an $\varepsilon > 0$ such that the measure of the level $k$ cells of $Y_i$ encoded by a walk containing $C$ is at least $\varepsilon \Hs(Y_i)$. Now, we can consider the first $nk$ levels of $Y$ for some $n$. By considering the walks between every $k$th level, it is easy to see that the measure of the level $nk$ cells encoded by a walk not containing $C$ is at most $(1 - \varepsilon)^n$. Taking the limit $n \to \infty$ completes the proof.
\end{proof}

\begin{Lemma} \label{multipart_inf_bounded}
    If there is a Lipschitz map $f : X \to Y$ with $\Hs(f(X)) > 0$, then $m_u$ is bounded on $[0, 1)$.
\end{Lemma}
\begin{proof}
    Let $L$ be the Lipschitz constant of $f$. First, notice that for any $L' > 0$, if $\Hs(f(X)) > 0$, then there is a walk $P$ ending in some $v \in V^X$ such that $r_P \le L'/L$ and $\Hs(f(S_P(X_v))) > 0$. Since $f \circ S_P$ is $L'$-Lipschitz, we may assume without loss of generality that $f$ is $\tilde L$-Lipschitz (possibly replacing $X$ by another attractor of $\Phi$).

    Note that $g_l$ is nonnegative, thus, by Doob's martingale convergence theorem, it converges for $\Hs$-a.e. point in $Y$. Consequently, there exists some $y \in f(X)$ such that $h_l(y) \to 1$ and $g_l(y)$ is convergent. Let $P$ be the infinite walk encoding $y$. By \cref{ae_cycle} we may further assume that $P$ goes through $u$ infinitely often.

    By \cref{preimage_decomp}, the set $f^{-1}(Y_{P|_l})$ is the union of cells congruent to sets of the form $r_Q X_v$, where $r_{\min}^X r_{P|_l} \le r_Q \le r_{P|_l}$. Scaling these cells by $r_{P|_l}^{-1}$, we obtain some $\tilde X_l \in \mathcal X$. Taking the restriction of $f$ onto these cells and also scaling $Y_{P|_l}$ by $r_{P|_l}^{-1}$ gives an $\tilde L$-Lipschitz map $\tilde f_l : \tilde X_l \to Y_{u_l}$, where $u_l$ is the endpoint of $P|_l$. Furthermore, we can see that $\Hs(\tilde X_l) = g_l(y) \Hs(Y_{u_l})$ and $\Hs(\tilde f_l(\tilde X_l)) = h_l(y) \Hs(Y_{u_l})$.

    It follows from the properties of $\tilde f_l$ that if $u_l=u$, then $m_u(h_l(y)) \le g_l(y) \Hs(Y_u)$. Since $u_l=u$ for infinitely many $l$ and $h_l(y) \to 1$, it follows from the monotonicity of $m_u$ that
    \[\sup_{t\in [0, 1)} m_u(t) \le \lim_{l\to \infty} g_{l}(y) \Hs(Y_u) < \infty,\]
    completing the proof.
\end{proof}

\begin{proof}[Proof of \cref{multipart_surj}]
    For $n > 0$, choose an $\tilde X_n \in \mathcal X$ and $\tilde f_n : \tilde X_n \to Y_u$ such that $\Hs(\tilde f_n(\tilde X_n)) \ge \zj[\big]{1 - \frac1n} \Hs(Y_u)$ and $\Hs(\tilde X_n) \le m_u\zj[\big]{1-\frac1n} + \frac1n$. It is clear from \cref{multipart_inf_bounded} that $\Hs(\tilde X_n)$ is bounded from above. Since $\Hs(X_v) > 0$ for every $v \in V^X$, and the scaling factors in $\mathcal X$ are bounded from below, this implies that the number of parts $p_n$ of $\tilde X_n = \bigcup_{i\le p_n}^* c_{n,i} X_{v_{n,i}}$ is bounded from above. Consequently, there is a sequence $n_k \to \infty$ such that $p_{n_k} = p$ for some constant $p$. By taking a subsequence, we may also assume that $v_{n_k,i}=v_i$ is constant and $c_{n_k, i} \to c_i$ for every $1 \le i \le p$.
    
    Let $\tilde X=\bigcup_{i \le p}^* c_i X_{v_i} \in \mathcal X$, and let $s_k : \tilde X \to \tilde X_{n_k}$ be the componentwise scaling. Note that the Lipschitz constant of $s_k$ converges to 1, so we may assume that it is always at most 2. Therefore, $\tilde f_{n_k} \circ s_k$ is $2\tilde L$-Lipschitz, so by the Arzelà--Ascoli theorem, we may further assume that $\tilde f_{n_k} \circ s_k \to \tilde f$ uniformly for some $\tilde f : \tilde X \to Y_u$. It is clear that $\tilde f$ is $\tilde L$-Lipschitz.

    Now take some nonempty relatively open subset $U \subseteq Y_u$, clearly $\Hs(U) > 0$. We know that $\Hs(\tilde f_{n_k}(\tilde X_{n_k})) \to \Hs(Y_u)$ by the choice of $\tilde f_{n_k}$, so for large enough $k$, $\tilde f_{n_k}(\tilde X_{n_k})$ intersects $U$. This is true for an arbitrary $U$, so the uniform convergence implies that the range of $\tilde f$ is dense in $Y_u$. Since $\tilde X$ is compact, it follows that $\tilde f$ is surjective.
    
    It remains to show that $\Hs(\tilde X)$ is minimal. We can see that $\Hs(\tilde X_{n_k}) \to \Hs(\tilde X)$. By the monotonicity of $m_u$, we have
    \[m_u(1) \ge m_u\zj*{1-\frac{1}{n_k}} \ge \Hs(\tilde X_{n_k}) - \frac{1}{n_k}.\]
    Taking the limit $k \to \infty$ yields the bound $m_u(1) \ge \Hs(\tilde X)$, thus $m_u(1)=\Hs(\tilde X)$. This shows that the pair $(\tilde X, \tilde f)$ minimizes $\Hs(\tilde X)$.
\end{proof}

From now on, let $\tilde f_u : \tilde X_u \to Y_u$ be an arbitrary $\tilde L$-Lipschitz surjection minimizing $\Hs(\tilde X_u)$.

\begin{Prop} \label{measure_lin}
    There exists a constant $\tilde c$ such that $(\tilde f_u)_*\zj{\Hs|_{\tilde X_u}} = \tilde c \Hs|_{Y_u}$ for every $u$.
\end{Prop}
\begin{proof}
    Let $\tilde c = \min_u \Hs(\tilde X_u)/\Hs(Y_u)$, and assume that $\Hs(\tilde X_u)=\tilde c \Hs(Y_u)$ for some $u$. Take some $v \in V^Y$ and a walk $P \in E^Y_{u,v,l}$. Applying \cref{preimage_decomp} to $\tilde f_u^{-1}(Y_P)$ (on each component of $\tilde X_u$) and scaling by $r_P^{-1}$, we obtain an $\tilde L$-Lipschitz surjection $\hat X_P \to Y_v$ for some $\hat X_P \in \mathcal X$. We can see that
    \begin{equation}
        \Hs(\tilde f_u^{-1}(Y_P)) = r_P^s \Hs(\hat X_P) \ge r_P^s \Hs(\tilde X_v) \ge
        r_P^s \tilde c \Hs(Y_v) = \tilde c \Hs(Y_P). \label{multipart_surj_linear.bound}
    \end{equation}
    Summing \cref{multipart_surj_linear.bound} over all walks of length $l$, we have
    \begin{align*}
        \Hs(\tilde X_u) &=
        \Hs\zj*{\bigcup_{v \in V^Y} \bigcup_{P\in E^Y_{u,v,l}} \tilde f_u^{-1}(Y_P)} =\\&=
        \sum_{v \in V^Y} \sum_{P\in E^Y_{u,v,l}} \Hs\zj[\big]{\tilde f_u^{-1}(Y_P)} \ge\\&\ge
        \tilde c \sum_{v \in V^Y} \sum_{P\in E^Y_{u,v,l}} \Hs(Y_P) = \tilde c \Hs(Y_u).
    \end{align*}
    However, we have $\Hs(\tilde X_u) = \tilde c \Hs(Y_u)$ by the choice of $u$, so we have equality in \cref{multipart_surj_linear.bound} for every walk $P$ starting from $u$. In particular, for any $v \in V^Y$, if $P$ is an arbitrary walk from $u$ to $v$, then we have $r_P^s \Hs(\tilde X_v) = r_P^s \tilde c \Hs(Y_v)$, thus $\Hs(\tilde X_v) = \tilde c \Hs(Y_v)$. It follows that $u$ can be chosen to be any vertex in $V^Y$. Finally, notice that the sets of the form $Y_P$ generate the Borel $\sigma$-algebra on $Y_u$, and since $(\tilde f_u)_*\Hs(Y_P) = \Hs(\tilde f_u^{-1}(Y_P)) = \tilde c \Hs(Y_P)$, this shows that $(\tilde f_u)_*\zj{\Hs|_{\tilde X_u}} = \tilde c \Hs|_{Y_u}$.
\end{proof}

The following proposition is a generalization of \cite[Proposition 5.1]{RuanXiao}. Note that this result is stronger even for IFSs, since we do not need to assume the homogeneity of $\Psi$.
\begin{Prop} \label{measure_lin_cover}
    There exists some constant $\varepsilon > 0$ such that for every cell $A$ of $\tilde X_u$, there is a cell $B$ of $Y_u$ such that $\diam B \ge \varepsilon \diam A$ and $\tilde f_u(A) \supseteq B$.
\end{Prop}
\begin{proof}
    Fix a constant $c > 0$ such that for every $v \in V^X$ and $u \in V^Y$, the inequality
    \[\Hs(X_v) \ge \frac{c}{(r_{\min}^Y)^s} \Hs(Y_u)\]
    holds. Let $A=X_{j, Q}$ be an arbitrary cell of $c_j X_{u_j} \subseteq \tilde X_u$, and let $v \in V^X$ be the endpoint of $Q$. We can see that
    \begin{equation}
        \Hs(A) = \Hs(X_v) (c_j r_Q)^s \ge c \zj*{\frac{c_j r_Q}{r_{\min}^Y}}^s \Hs(Y_{u'})\label{measure_lin_cover.c_bound}
    \end{equation}
    for every $u' \in V^Y$.

    Suppose for contradiction that no $\varepsilon$ with the needed property exists. Let $k$ be a positive integer. Let $\rho_Y=\min\ha{\Hs(Y_u)/\Hs(Y_v)}{u,v\in V_Y}$, and choose a large integer $n$ such that $\tilde c \zj[\big]{1-\rho_Y (r_{\min}^Y)^{sk}}^{n} < c$. Let $m$ be large enough that $(r_{\max}^Y)^m \le r_{\min}^Y$, where $r_{\max}^Y$ is the largest contraction ratio of $\Psi$. Choose $\varepsilon > 0$ such that $\varepsilon \diam X_u \le (r_{\min}^Y)^{m + k n} \diam Y_v$ for every $u \in V^X$ and $v \in V^Y$. By the hypothesis, there is some $u \in V^Y$ and a cell $A=X_{j,Q}$ of $\tilde X_u$ such that $\tilde f_u(A)$ does not contain a cell of diameter at least $\varepsilon \diam A$.
    
    Clearly, we may assume that $c_j r_Q < 1$. Let $Y_P$ be a cell of $Y_{u}$ intersecting $\tilde f_u(A)$ such that $c_j r_Q < r_P \le c_j r_Q/r_{\min}^Y$, and let $v \in V^Y$ be the endpoint of $P$. Since $d(Y_P, Y_u \setminus Y_P) \ge \delta r_P > \delta c_j r_Q \ge \tilde L \diam A$, it is clear that $\tilde f_u(A) \subseteq Y_P$. Using \cref{measure_lin_cover.c_bound}, we can see that
    \begin{multline*}
        \tilde c \Hs(\tilde f_u(A)) = \Hs\zj[\big]{\tilde f_u^{-1}(\tilde f_u(A))} \ge \Hs(A)
        \ge\\\ge c \Hs(Y_{v}) \zj*{\frac{c_j r_Q}{r_{\min}^Y}}^s \ge c \Hs(Y_{v}) r_P^s = c \Hs(Y_P).
    \end{multline*}
    By the choice of $n$, it follows that
    \begin{equation} \label{measure_lin_cover.image_bound}
        \Hs(\tilde f_u(A)) > \zj[\big]{1-\rho_Y(r_{\min}^Y)^{sk}}^{n} \Hs(Y_P).
    \end{equation}
    We will now show that there is some $|P| + m \le l \le |P| + m + (n-1) k$ and a level $l$ cell $B \subseteq Y_P$ such that $\tilde f_u(A)$ intersects every level $l + k$ cell contained in $B$. Assume for contradiction that no such $B$ exists. If we take a level $|P| + m + (n-1)k$ cell $B$, then $\tilde f_u(A)$ is disjoint from at least one level $|P| + m + nk$ cell within $B$ which is disjoint from $\tilde f_u(A)$. The measure of this cell is at least $\rho_Y(r_{\min}^Y)^{sk} \Hs(B)$, thus $\Hs(\tilde f_u(A) \cap B) \le \zj[\big]{1 - \rho_Y(r_{\min}^Y)^{sk}} \Hs(B)$. We can continue this argument by induction to show that if $B$ is a level $|P| + m + (n-i) k$ cell for $1 \le i \le n$, then
    \[\Hs(\tilde f_u(A) \cap B) \le \zj[\big]{1 - \rho_Y(r_{\min}^Y)^{sk}}^i \Hs(B).\]
    Setting $i=n$, then summing over the level $|P| + m$ cells in $Y_P$, we have
    \[\Hs\zj[\big]{\tilde f_u(A)} = \Hs\zj[\big]{\tilde f_u(A) \cap Y_P} \le \zj[\big]{1 - \rho_Y(r_{\min}^Y)^{sk}}^n \Hs(Y_P),\]
    which contradicts \cref{measure_lin_cover.image_bound}.
    
    We have thus shown the existence of the cell $B=Y_{P'}$. Let $w\in V^Y$ be the endpoint of $P'$. We have
    \begin{multline*}
        \diam B = r_{P'} \diam Y_{w} \ge r_P (r_{\min}^Y)^{m + k n} \diam Y_{w} \ge\\\ge c_j r_Q (r_{\min}^Y)^{m + k n} \diam Y_{w} \ge \varepsilon \diam A.
    \end{multline*}
    It follows from the choice of $A$ that $\tilde f_u(A)$ does not cover $B$. Using \cref{preimage_decomp}, we can see that $\tilde f_u^{-1}(B)$ contains a cell $X_{j', Q'}$ not contained in $A$ such that $r_{\min}^X r_{P'} \le c_{j'} r_{Q'} < r_{P'}$.

    Since
    \[c_{j'} r_{Q'} < r_{P'} \le r_P (r_{\max}^Y)^m \le r_P r_{\min}^Y \le c_j r_Q,\]
    this is possible only if $A \cap X_{j', Q'} = \emptyset$. Consequently, if $w \in V^Y$ is the endpoint of $P'$, then by \cref{measure_lin_cover.c_bound}, we have
    \begin{align*}
        \Hs\zj[\big]{\tilde f_u^{-1}(B) \cap A} &\le
        \Hs\zj[\big]{\tilde f_u^{-1}(B)} - \Hs(X_{j', Q'}) \le\\&\le
        \tilde c \Hs(B) - \frac{c(c_{j'} r_{Q'})^s}{(r_{\min}^Y)^s} \Hs(Y_w) \le\\&\le
        \tilde c \Hs(B) - c\zj*{\frac{r_{\min}^X}{r_{\min}^Y}}^s r_{P'}^s \Hs(Y_w) =\\&=
        \tilde c \Hs(B) - c\zj*{\frac{r_{\min}^X}{r_{\min}^Y}}^s \Hs(B) = 
        \hat c \Hs(B),
    \end{align*}
    where $\hat c = \tilde c - c \zj[\big]{r_{\min}^X/r_{\min}^Y}^s < \tilde c$. We may assume that $\hat c \ge 0$, otherwise, we already have a contradiction from the previous inequality.

    If we consider the restriction of $\tilde f_u$ to $A$ and scale both $A$ and $B$ by $r_{P'}^{-1}$, then by \cref{preimage_decomp}, we obtain an $\tilde L$-Lipschitz map $\hat f_k : \hat X_k \to Y_{w_k}$, where $\hat X_k \in \mathcal X$, $\Hs(\hat X_k) \le \hat c \Hs(Y_{w_k})$ and $\hat f_k(\hat X_k)$ intersects every level $k$ cell of $Y_{w_k}$. As in the proof of \cref{multipart_surj}, we can take a subsequence $k_i$ such that $w_{k_i}=w$ is constant, the number of parts in the decomposition of $\hat X_{k_i}$ is constant, furthermore, the vertices and scaling factors in $\hat X_{k_i}$ also converge. Using the Arzelà-Ascoli theorem, we obtain some $\hat f : \hat X \to Y_{w}$ such that $\Hs(\hat X) \le \hat c \Hs(Y_w)$ and $\hat f(\hat X)$ intersects every level $k$ cell for every $k$. But this implies the surjectivity of $\hat f$, which contradicts the minimality of $\tilde c$.
\end{proof}

\section{Proof of the commensurability} \label{sec:surj_to_comm}

In this section, we will prove the implication $\ref{equiv.surj} \Rightarrow \ref{equiv.commensurable}$ of \cref{equiv}.

Let $X=X_1$ and $Y=Y_1$ be the attractors of the dust-like strongly connected GD-IFSs $\Phi = (V^X, E^X, S^X)$ and $\Psi = (V^Y, E^Y, S^Y)$, where $V^X = \{1, \dots, \#V^X\}$ and $V^Y = \{1, \dots, \#V^Y\}$. We will assume that $\dimH X = \dimH Y = s > 0$ and $\Phi$ has commensurable cycles. By \cref{commensurable_to_homogeneous}, we may assume that each $S^X_e$ has similarity ratio $r$ for some $r\in (0, 1)$ and the GCD of cycle lengths in $(V^X, E^X)$ is 1. Let $f : X \to Y$ be an $L$-Lipschitz map with the property that $\Hs(f(X)) > 0$.

Recall that the constant $\delta > 0$ was chosen such that $d(Y_P, Y \setminus Y_P) \ge \delta r_P$ for any walk $P$. Let $D = \max\ha{\diam X_i}{i \in V^X}$. For a walk $P$ in $(V^Y, E^Y)$ starting from $1$, define $l_P$ to be the least nonnegative integer such that $L D r^{l_P} < \delta r_P$.

\begin{Prop}
    For every walk $P$ in $(V^Y, E^Y)$ starting from $1$, there is a vector $\vv_P \in \N^{V^X}$ such that
    \[f_*\Hs(Y_P) = r^{s l_P} \vv_P\mv\]
\end{Prop}
\begin{proof}
    Let $Q \in E^X_{1,i,l_P}$ for some $i$. It is clear that $\diam X_Q \le D r_Q = D r^{l_P}$, thus, we have $\diam f(X_Q) \le L D r^{l_P} < \delta r_P$. Therefore, either $f(X_Q) \subseteq Y_P$ or $f(X_Q) \cap Y_P = \emptyset$. Consequently,
    \begin{align*}
        f_*\Hs(Y_P) &=
        \sum_{i \in V^X} \sum_{Q \in E^X_{1,i,l_P}} \Hs\zj[\big]{X_Q \cap f^{-1}(Y_P)} =\\&=
        \sum_{i \in V^X} \sum_{\substack{Q \in E^X_{1,i,l_P} \\ f(X_Q) \subseteq Y_P}} \Hs(X_Q) =
        \sum_{i \in V^X} v_i r^{s l_P} \Hs(X_i),
    \end{align*}
    where $v_i = \#\ha{Q \in E^X_{1,i,l_P}}{f(X_Q) \subseteq Y_P}$.
\end{proof}

\begin{Prop} \label{vec_finite}
    For every $c > 0$, the set $\ha{\vv_P}{f_*\Hs(Y_P) \le c \Hs(Y_P)}$ is finite.
\end{Prop}
\begin{proof}
    Since the maps $S^Y_e$ are contractions, it is clear that $l_P > 0$ apart from finitely many cases. By the minimality of $l_P$, we have $L D r^{l_P - 1} \ge \delta r_P$. Let $i \in V^Y$ be the last vertex of $P$. If $f_*\Hs(Y_P) \le c \Hs(Y_P)$, then
    \[\vv_P \mv = \frac{f_*\Hs(Y_P)}{r^{s l_P}} \le c \frac{\Hs(Y_P)}{r^{s l_P}} =
    c \Hs(Y_i) \zj*{\frac{r_P}{r^{l_P}}}^s \le c \Hs(Y_i) \zj*{\frac{L D}{\delta r}}^s.\]
    Note that the right-hand side does not depend on $P$. Since $i$ has finitely many values, it is bounded, so $\vv_P \mv$ is also bounded. Since $\vv_P$ has nonnegative components and $\mv$ has positive components, it is clear that each component of $\vv_P$ is bounded. Since the components of $\vv_P$ are integers, each component has finitely many possible values, thus $\vv_P$ itself has finitely many values.
\end{proof}

Recall that $\mathcal B_{l}$ is the $\sigma$-algebra generated by the level $l$ cells in $Y$. Like before, we may assume that $\Hs(Y) = 1$, which implies that
\[g_l|_{Y_P}\equiv \frac{f_*\Hs(Y_P)}{\Hs(Y_P)}\]
is a martingale.

The next step is to show that the change of $g_l$ cannot be arbitrarily small. An analogous property was used by Falconer and Marsh \cite{FalconerMarsh} in the case of bi-Lipschitz maps between dust-like self-similar sets.
\begin{Prop} \label{diff_bound}
    For every $c > 0$, there is an $\varepsilon > 0$ such that if $g_l(y) \le c$, then either $g_{l + 1}(y) = g_l(y)$ or $|g_{l + 1}(y) - g_l(y)| \ge \varepsilon$.
\end{Prop}
\begin{proof}
    Let $y\in Y$. There is a walk $P \in E^Y_{1,i,l}$ and an edge $e\in E^Y_{i,j}$ such that $y \in Y_{P e}$.
    We can see that
    \begin{align*}
        \abs{g_{l + 1}(y) - g_{l}(y)} &= 
        \abs*{\frac{f_*\Hs(Y_{P e})}{\Hs(Y_{P e})} - \frac{f_*\Hs(Y_P)}{\Hs(Y_P)}} =\\&=
        \abs*{\frac{r^{s l_{P e}} \vv_{P e} \mv}{\Hs(Y_j) (r_P r_e)^s} - \frac{r^{s l_P} \vv_{P} \mv}{\Hs(Y_i) r_P^s}} =\\&=
        \zj*{\frac{r^{l_{P}}}{r_P}}^s \abs*{\frac{r^{s (l_{P e} - l_P)} \vv_{P e} \mv}{\Hs(Y_j) r_e^s} - \frac{\vv_{P} \mv}{\Hs(Y_i)}}
    \end{align*}

    Apart from finitely many cases, $l_P > 0$. The minimality of $l_P$ implies that $L D r^{l_P - 1} \ge \delta r_P$, hence
    \[\zj*{\frac{r^{l_P}}{r_P}}^s \ge \zj*{\frac{\delta r}{L D}}^s > 0.\]
    This shows that the first factor is bounded from below.
    
    We may assume that $\abs{g_{l + 1}(y) - g_{l}(y)} \le 1$, otherwise we are done.
    This implies that $g_{l + 1}(y) \le c + 1$. By \cref{vec_finite}, both $\vv_P$ and $\vv_{P e}$ have finitely many possible values. Furthermore, if $l$ is chosen so that $r^l$ is smaller than $r_e$ for every $e \in E^Y$, then clearly $l_{P e} \le l_P + l$. This shows that $0 \le l_{P e} - l_P \le l$, so $l_{P e} - l_P$ also has finitely many values. Since $i$, $j$ and $e$ all have finitely many values, this shows that the second factor has finitely many values. Therefore, it is either zero or bounded from below by a positive constant. Consequently, $\abs{g_{l + 1}(y) - g_{l}(y)}$ is either zero or bounded from below by some $\varepsilon > 0$.
\end{proof}


\begin{proof}[Proof of $\ref{equiv.surj} \Rightarrow \ref{equiv.commensurable}$ when $R(\Phi)$ is commensurable]
    Since $g_l$ is a nonnegative martingale, it converges almost everywhere. Therefore, we can choose a point $y\in f(X)$ such that $g_l(y)$ converges, furthermore, by \cref{ae_cycle}, $y$ can be chosen such that the infinite walk $P$ encoding it contains every possible cyclic walk in $(V^Y, E^Y)$.

    Since the sequence $\{g_l(y)\}_l$ is convergent, it is clearly bounded, so by \cref{diff_bound}, there is a $\varepsilon > 0$ such that $g_{l+1}(y) = g_l(y)$ or $|g_{l+1}(y) - g_l(y)| \ge \varepsilon$ for every $l$. However, the convergence implies that $|g_{l+1}(y) - g_l(y)| < \varepsilon$ for sufficiently large $l$, therefore, the sequence eventually becomes constant.

    For $l \ge 0$, let $P|_l$ denote the walk formed by the first $l$ edges of $P$. Clearly $y \in Y_{P|_l}$. By the assumption that $y \in f(X)$, the set $f^{-1}(Y_{P|_l})$ is nonempty. Furthermore, it is an open subset of $X$ (since $Y_{P|_l}$ is open in $Y$), hence $\Hs(f^{-1}(Y_{P|_l})) > 0$. It follows that $g_l(y) > 0$ for every $l$.

    Fix a cycle $C$ of length $k$ in $(V^Y, E^Y)$, and also fix some starting vertex in $C$. Denote by $C^n$ the cyclic walk obtained by repeating $C$ $n$ times. We know that $P$ contains $C^n$ for every $n$, so there is an $l_0$ such that for every $0 \le m < n$, the edges of $P$ from $l_0 + m k$ to $l_0 + (m+1) k - 1$ are a copy of $C$. Since $n$ is arbitrarily large, we may assume that $l_0$ is large enough that $g_l(y)$ is constant for $l \ge l_0$.
    
    Let $P_m = P|_{l_0 + m k}$. We know from \cref{vec_finite} that the vector $\vv_{P_m}$ has finitely many possible values. It follows that if $n$ is large enough, then there exist $0 \le m_1 < m_2 < n$ such that $\vv_{P_{m_1}} = \vv_{P_{m_2}}$. We can check that
    \begin{align*}
        g_{l_0 + m_2 k}(y) &=
        \frac{f_*\Hs(Y_{P_{m_2}})}{\Hs(Y_{P_{m_2}})} =
        \frac{r^{s l_{P_{m_2}}} \vv_{P_{m_2}}\mv}
            {r_C^{s(m_2 - m_1)}\Hs(Y_{P_{m_1}})} =\\&=
        \frac{r^{s (l_{P_{m_2}} - l_{P_{m_1}})} r^{s l_{P_{m_1}}} \vv_{P_{m_1}}\mv} {r_C^{s(m_2 - m_1)}\Hs(Y_{P_{m_1}})} =
        \zj*{\frac{r^{l_{P_{m_2}} - l_{P_{m_1}}}}{r_C^{m_2 - m_1}}}^s g_{l_0 + m_1 k}(y).
    \end{align*}
    Since $g_{l_0 + m_2 k}(y) = g_{l_0 + m_1 k}(y) > 0$, it follows that $r^{l_{P_{m_2}} - l_{P_{m_1}}} = r_C^{m_2 - m_1}$, or equivalently
    \[r_C = r^{\frac{l_{P_{m_2}} - l_{P_{m_1}}}{m_2 - m_1}}.\]
    As $C$ was arbitrary, this shows that every element of $R(\Psi)$ is a rational power of $r$, which completes the proof.
\end{proof}

Finally, we will prove the implication $\ref{equiv.surj} \Rightarrow \ref{equiv.commensurable}$ assuming the commensurability of $R(\Psi)$ rather than $R(\Phi)$.

\begin{proof}[Proof of $\ref{equiv.surj} \Rightarrow \ref{equiv.commensurable}$ when $R(\Psi)$ is commensurable]
    Assume that $\Psi$ has commensurable cycles. By \cref{commensurable_to_homogeneous}, we may assume that $\Psi$ is homogeneous with common ratio $r$.
    
    Fix some cycle $C$ in $(V^X, E^X)$. We may consider $C$ as a walk from some $v \in V^X$ to itself. Let $u \in V^Y$ be arbitrary, and take $A_0$ to be an arbitrary cell of $\tilde X_u$ that is similar to $X_v$ (such an $A_0$ exists by the strong connectivity). For $n > 0$, let $A_n$ be the cell of $A_0$ along the walk $C^n$.
    
    By \cref{multipart_surj}, we can take an $\tilde L$-Lipschitz surjection $\tilde f_u : \tilde X_u \to Y_u$ minimizing $\Hs(\tilde X_u)$ for an arbitrary $u \in V^Y$. By \cref{measure_lin_cover}, there is some cell $B_n \subseteq \tilde f_u(A_n)$ of $Y_u$ such that $\diam B_n \ge \varepsilon \diam A_n$ for some constant $\varepsilon > 0$. Let $B_n=Y_P$, where $P$ is a walk ending in some $u_n$. Assuming for contradiction that $\tilde f_u^{-1}(B_n) \not\subseteq A_n$, we have $\Hs(\tilde f_u^{-1}(B_n) \cap A_n) < \tilde c \Hs(B_n)$. It follows from \cref{preimage_decomp} that by restricting $\tilde f_u$ to $\tilde f_u^{-1}(B_n) \cap A_n$ and scaling everything by $r_P^{-1}$, we obtain an $\tilde L$-Lipschitz surjection $\hat f : \hat X \to Y_{u_n}$ with $\hat X \in \mathcal X$ such that $\Hs(\hat X) < \tilde c \Hs(Y_{u_n})$. This contradicts the minimality of $\tilde c$, therefore, $\tilde f_u^{-1}(B_n) \subseteq A_n$.
    
    Using \cref{preimage_decomp}, we can see that $\tilde f_u^{-1}(B_n)$ is the disjoint union of some cells of diameter at least $c \diam B_n \ge c \,\varepsilon \diam A_n$ for some constant $c > 0$. Since these cells are all contained in $A_n$, there are finitely many possibilities up to similarity. In particular, it follows that $\Hs(\tilde f_u^{-1}(B_n))/\Hs(A_n)$ has finitely many possible values. Consequently, we can choose $n_1 < n_2$ such $B_{n_1}$ and $B_{n_2}$ are scaled copies of the same $Y_w$, and
    \[
    \frac{\Hs\zj[\big]{\tilde f_u^{-1}(B_{n_1})}}{\Hs(A_{n_1})} = \frac{\Hs\zj{\tilde f_u^{-1}(B_{n_2})}}{\Hs(A_{n_2})}.\]
    It follows that if $B_{n_i}$ is a level $m_i$ cell of $Y_u$, then
    \[r_C^{s(n_2 - n_1)} = \frac{\Hs(A_{n_2})}{\Hs(A_{n_1})} = \frac{\Hs(\tilde f_u^{-1}(B_{n_2}))}{\Hs(\tilde f_u^{-1}(B_{n_1}))} = \frac{\tilde c \Hs(B_{n_2})}{\tilde c \Hs(B_{n_1})} = \frac{r^{sm_2} \Hs(Y_w)}{r^{sm_1} \Hs(Y_w)} = r^{s(m_2 - m_1)}.\]
    Therefore, $r_C = r^{(m_2 - m_1)/(n_2-n_1)}$, which completes the proof.
\end{proof}

\section{Construction of the surjection} \label{sec:comm_to_surj}

To construct the surjection, we will first prove the following variant of Xi and Xiong's mass decomposition lemma \cite[Lemma 3.4]{XiXiong}.

\begin{Lemma} \label{mass_decomp}
    Let $\Phi$ be a primitive GD-IFS with an attractor $X$, and assume that $\Hs(X) = 1$. Let $p = r^s$, and let $\Omega \subseteq \Z[p] \cap (0, \infty)$ be a finite set. Denote by $\mathcal B_k$ the algebra of sets generated by the level $k$ cells of $X$. Then there exists a positive integer $k_0$ such that for any set $B \in \mathcal B_k$ and positive numbers $b_1, \dots, b_m \in p^k\Omega$ such that $\Hs(B) = b_1 + \dots + b_m$ there exists a partition $\{B_1, \dots, B_m\}$ of $B$ such that $\Hs(B_i) = b_i$ and $B_i \in \mathcal B_{k + k_0}$ for every $i$.
\end{Lemma}

The proof will be essentially the same as in \cite{XiXiong} with some modifications for the graph-directed setting. We can reuse the following technical lemma without modification:
\begin{Lemma}[{\cite[Lemma 5.4]{XiXiong}}] \label{sum_decomp}
    Let $U$ and $V$ be finite sets of positive numbers. There exists a constant $M$ such that for any two finite sequences $(a_i)_{i\in \mathcal I}$ and $(b_j)_{j\in \mathcal J}$ such that
    \begin{enumerate}
        \item $a_i \in U$ and $b_j \in V$ for $i\in \mathcal I$ and $j \in \mathcal J$;
        \item $\sum_{i\in \mathcal I} a_i = \sum_{j \in \mathcal J} b_j > M$;
    \end{enumerate}
    there are partitions $\mathcal I = \mathcal I_1 \cup \mathcal I_2$ and $\mathcal J= \mathcal J_1 \cup \mathcal J_2$ such that
    \[\sum_{i \in \mathcal I_1} a_i = \sum_{j \in \mathcal J_1} b_j > 0, \qquad
    \sum_{i \in \mathcal I_2} a_i = \sum_{j \in \mathcal J_2} b_j > 0.\]
\end{Lemma}

\begin{Lemma} \label{int_decomp}
    Suppose that $\Hs(X_1) = 1$, and let $x\in \Z[p]$. For sufficiently large $k$, $x$ can be written in the form $p^k \vv \mv$ for some row vector $\vv \in \Z^n$.
\end{Lemma}
\begin{proof}
    We can write $x = \sum_{i=0}^k c_i p^i$ with $c_i \in \Z$. By adding zero terms, the value of $k$ can be increased arbitrarily. Setting $\uv=(1, 0, \dots, 0)$, we have $\uv \mv = 1$. Notice that
    \begin{align*}
        x &= \sum_{i=0}^k c_i p^i = \sum_{i=0}^k c_i p^i \uv \mv =
        \sum_{i=0}^k c_i p^i \uv (p A)^{k - i} \mv =\\&=
        \sum_{i=0}^k c_i p^k \uv A^{k - i} \mv =
        p^k \zj*{\uv \sum_{i=0}^k c_i A^{k - i}} \mv.
    \end{align*}
    Since $A$ is an integer matrix, it is clear that $\vv := \uv \sum_{i=0}^k c_i A^{k - i}$ is an integer vector.
\end{proof}

\begin{proof}[Proof of \cref{mass_decomp}]
    We can write $B=\cup_{j=1}^l B_j$, where the $B_j$ are (pairwise distinct) level $k$ cells of $X$. Each $B_j$ is a copy of some $X_{i_j}$ scaled by $r^k$, therefore
    \[\Hs(B_j) = r^{ks} \Hs(X_{i_j}) = p^k \Hs(X_{i_j}).\]
    It follows that
    \[\sum_{i=1}^{m} p^{-k} b_i = p^{-k} \Hs(B) = \sum_{j=1}^l \Hs(X_{i_j}).\]
    Notice that $p^{-k} b_i \in \Omega$ for every $i$. Let $M$ be the constant from \cref{sum_decomp} with $U=\Omega$ and $V = \{\Hs(X_1), \dots, \Hs(X_n)\}$. We may assume that $p^{-k} \Hs(B) \le M$, otherwise, we can partition the sums using \cref{sum_decomp} and proceed by induction on $m$ (keeping $k_0$ constant).

    Let $\vv = (v_1, \dots, v_n)$, where $v_i$ is the number of scaled copies of $X_i$ among the level $k$ cells of $B$. Clearly $\vv \mv = p^{-k} \Hs(B) \le M$. Since the components of $\vv$ are nonnegative integers, each component has finitely many possible values. It follows that $\vv$ has finitely many possible values, therefore, it is enough to find some $k_0$ for a fixed $\vv$. We can also see that $m$ is bounded from above, since $m \min \Omega \le \sum_{i=1}^m p^{-k} b_i \le M$. Consequently, the vector $p^{-k}(b_1, \dots, b_m)$ has finitely many possible values, so $k_0$ may depend also on $p^{-k} b_i$.

    Now fix some $\vv$ and $b_1, \dots, b_m$. Since $p^{-k} b_i \in \Z[p]$, it follows from \cref{int_decomp} that there is a $k_1$ such that for $1 \le i < m$, there is a vector $\vv_i$ with the property $p^{-k} b_i = p^{k_1} \vv_i \mv$. We choose $\vv_m$ so that $\sum_{i=1}^m \vv_i = \vv A^{k_1}$. Notice that
    \[\sum_{i=1}^m p^{-k} b_i = p^{-k} \Hs(B) = \vv \mv = p^{k_1} \vv A^{k_1} \mv =
    \sum_{i=1}^m p^{k_1} \vv_i \mv,\]
    hence the formula $p^{-k} b_i = p^{k_1} \vv_i \mv$ holds even for $i=m$.

    Let $\qv$ be the left eigenvector of $A$ satisfying $\qv A = p^{-1} \qv$ and $\qv \mv = 1$. Since $p^{-1}$ is the Perron root of $A$, we have $(p A)^l \to \mv \qv$ as $l \to \infty$ (see \cite[Theorem 8.5.1]{PF1}). It follows that
    \[\vv_i (p A)^l \to \vv_i \mv \qv = p^{-k_1-k} b_i \qv > 0.\]
    If $l$ is large enough, then $\vv_i (p A)^l>0$, or equivalently, $\vv_i A^l > 0$. For some $l=k_2$, the inequality $\vv_i A^{k_2}> 0 $ holds for every $i$. Set $k_0 = k_1 + k_2$ and $\uv_i = \vv_i A^{k_2}$. We can check that
    \[\sum_{i=1}^m \uv_i = \zj[\bigg]{\sum_{i=1}^m \vv_i} A^{k_2} = \vv A^{k_1 + k_2} = \vv A^{k_0}.\]
    Since $\vv$ gives the number of level $k$ cells of $B$, it is easy to see from the definition of $A$ that $\vv A^{k_0}$ gives the number of level $(k+k_0)$ cells. Since $\sum_{i=1}^m \uv_i = \vv A^{k_0}$ and each $\uv_i$ is a vector of positive integers, we can partition the level $(k+k_0)$ cells of $B$ according to the vectors $\uv_i$. Let $B_1, \dots, B_m$ be the unions of these cells. We can see that
    \[\Hs(B_i) = p^{k + k_0} \uv_i \mv = p^{k + k_1} \vv_i (p A)^{k_2} \mv =
    p^{k + k_1} \vv_i \mv = b_i.\]
    This shows that the partition $\{B_1, \dots, B_m\}$ has the required properties.
\end{proof}

\begin{Lemma} \label{n_copies_level}
    Let $X$ be an attractor of a dust-like strongly connected GD-IFS $\Phi$. If $X$ contains more than one point, then for every $m$ there is a $k$ such that at least $m$ level $k$ cells of $X$ are scaled copies of $X$.
\end{Lemma}
\begin{proof}
    First, consider the case that no vertex has more than one outgoing edge. Since the graph was assumed to be strongly connected, this is possible only if it is a single directed cycle (or a single vertex with no edges, but in this case, $X$ is empty). For each vertex $v$, let $x_v$ be the unique fixed point obtained by composing the similarities along the cycle starting from $v$. We will check that the sets $\{x_v\}$ are the attractors of $\Phi$. Let $e \in E_{u,v}$ be an arbitrary edge, let $P$ be the path from $v$ to $u$. Then $S_P(S_e(x_v)) = x_v$, therefore,
    \[S_e(S_P(S_e(x_v))) = S_e(x_v).\]
    By the uniqueness of the fixed point of $S_e \circ S_P$, it follows that $x_u=S_e(x_v)$, which shows that the sets $\{x_v\}$ are indeed the attractors. However, by the uniqueness of the attractors, $X=\{x_v\}$ for some $v$, contradicting the assumption.
    
    We may thus assume that there is a vertex $v$ with at least two outgoing edges. Since the graph is strongly connected, this implies that there are cyclic walks $P$ and $Q$ starting from $v$ such that their first edges are different. If $|P|=a$ and $|Q|=b$, then the cyclic walks $P^b$ and $Q^a$ are different cyclic walks of length $ab$. It follows that there are at least $2^l$ different cyclic walks of length $abl$ starting from $v$. Let $u$ be the vertex corresponding to $X$. We can insert the $2^l$ cyclic walks between a path from $u$ to $v$ and a path from $v$ to $u$. These cyclic walks all have the same length, so we are done if $l$ is chosen such that $2^l \ge m$.
\end{proof}

\begin{proof}[Proof of $\ref{equiv.commensurable} \Rightarrow \ref{equiv.surj}$]
    By \cref{commensurable_to_homogeneous}, we may assume that $\Phi$ and $\Psi$ are both homogeneous with similarity ratios $r^X$ and $r^Y$. By the commensurability assumption, there exist positive integers $a$ and $b$ such that $(r^X)^a = (r^Y)^b$. Using \cref{homogeneous_power}, we may replace the ratios by powers, so from now on, we will assume that $\Phi$ and $\Psi$ have the same similarity ratio $r$.

    We may assume that $\Hs(X) = \Hs(Y) = 1$. Let $\Omega_0$ be the set of the Hausdorff measures of the attractors of $\Psi$. Let $p=r^s$. It follows from \cref{vector_in_ring} that $\Omega_0 \subseteq \Q[p]$. Since $\Omega_0$ is finite, there exists some integer $c > 0$ such that $\Omega := c \Omega_0 \subseteq \Z[p]$. Let $k_0$ be the constant from \cref{mass_decomp} for $X$ and $\Omega$. By the assumption $s > 0$, $X$ is not a single point, so by \cref{n_copies_level}, there is an $l_{-1}$ such that there are at least $c$ scaled copies of $X$ among the level $l_{-1}$ cells of $X$. Let $X'$ be the union of the $c$ copies.

    Let $l_m = l_{-1} + (m + 1) k_0$, and let $\mathcal C^X_m$ be the set of level $l_m$ cells contained in $X'$. Similarly, let $\mathcal C^Y_m$ be the set of level $m k_0$ cells of $Y$. Our goal is to construct surjective maps $f_m : \mathcal C^X_m \to \mathcal C^Y_m$ for $m \ge 0$ such that whenever $B \in \mathcal C^X_m$, $B' \in \mathcal C^X_{m+1}$ and $B'\subseteq B$, we have $f_{m+1}(B') \subseteq f_m(B)$. Assuming that we have constructed such maps, each point $x \in X'$ is contained in some decreasing sequence $B_0 \supseteq B_1 \supseteq \dots$ such that $B_m \in \mathcal C^X_m$. Then the sets $f_m(B_m)$ are also decreasing by the construction of $f_m$, and since they are closed and their diameters converge to zero, their intersection is a single point, which we will call $f(x)$. We have thus defined a map $f : X' \to Y$. To see that $f$ is Lipschitz, notice that if some $x$ and $x'$ are contained in $B \in \mathcal C^X_m$, but they are in different level $l_{m+1}$ cells, then their distance is at least $\delta r^{l_m}$ for some constant $\delta > 0$. Their images are both contained in the level $m k_0$ cell $f_m(B)$, which has diameter at most $D r^{m k_0}$ for some constant $D$. This shows that $f$ is Lipschitz with Lipschitz constant at most $D r^{-l_0} / \delta$. It is easy to see from the surjectivity of $f_m$ that the image of $f$ is dense, so by the compactness of $X'$, it is the whole $Y$. Finally, since $X' \subseteq X$ is clopen, $f$ stays Lipschitz if we set it to some constant point in $Y$ on $X \setminus X'$.

    We will now construct the functions $f_m$ with the following additional property: for every $C \in \mathcal C^Y_m$, we have
    \[\Hs(f_m^{-1}(C)) = c p^{l_{-1}} \Hs(C).\]
    This clearly guarantees the surjectivity of $f_m$. We can also check that this holds for $f_0$, which must map every cell in $\mathcal C^X_0$ to $Y$. Assume that we have already constructed $f_m$. Let $C \in \mathcal C^Y_m$, and let $C_1, \dots, C_n$ be the sets in $\mathcal C^Y_{m+1}$ contained in $C$. We want to construct a partition $B_1, \dots, B_n$ of $B=f_m^{-1}(C)$ such that $B_i \in \mathcal B_{l_{m + 1}}$ and
    \[\Hs(B_i)=c p^{l_{-1}} \Hs(C_i)\]
    for every $i$. Since $\Hs(C_i) \in p^{(m+1) k_0} \Omega_0$, we can see that
    \[\Hs(B_i) \in c p^{l_{-1} + (m+1) k_0} \Omega_0 = p^{l_m} \Omega.\]
    So such a partition indeed exists by \Cref{mass_decomp}. We can now define $f_{m+1}$ by mapping each level $l_{m+1}$ cell of $B_i$ to $C_i$. It is easy to check that this map has all the required properties.
\end{proof}

\section*{Acknowledgements}

The author would like to thank his supervisors Márton Elekes and Tamás Keleti for their guidance and many helpful discussions. The author is also grateful to Siqi Wang for her comments on an earlier version of the paper.

The author was supported by the National Research, Development and Innovation Office -- NKFIH, grants nos. 146922 and 152822.

\section*{Declaration of generative AI and AI-assisted technologies}

OpenAI's ChatGPT was used to obtain feedback on earlier drafts of this manuscript, including possible mathematical inaccuracies, missing details, inconsistencies in constants and notation, typographical errors, and clear language issues. The mathematical results, arguments, proofs, and the final text of the manuscript are the author's own. The author reviewed and edited all AI-assisted feedback and takes full responsibility for the content of the paper. This declaration was drafted with the assistance of ChatGPT and edited by the author.

\printbibliography

\end{document}